%% file: Cauchy.tex
\documentclass[12pt,oneside,reqno]{amsart}

\usepackage{amsmath,amsthm,amsfonts,mathabx,tikz,enumitem,stmaryrd,url,mathtools}
\usepackage{hyperref}
\hypersetup{
	colorlinks,
	citecolor=black,
	filecolor=black,
	linkcolor=black,
	urlcolor=black
}

\usetikzlibrary{matrix,arrows,decorations.pathmorphing}

\newtheorem{defn}{Definition}[section]
\newtheorem{prop}{Proposition}[section]
\newtheorem{cor}{Corollary}[section]
\newtheorem{remark}{Remark}[section]

\newtheorem{example}{Example}[section]

\makeatletter
\newcommand{\pto}{}
\newcommand{\pgets}{}

\DeclareRobustCommand{\pto}{\mathrel{\mathpalette\p@to@gets\to}}
\DeclareRobustCommand{\pgets}{\mathrel{\mathpalette\p@to@gets\gets}}

\newcommand{\p@to@gets}[2]{%
  \ooalign{\hidewidth$\m@th#1\mapstochar\mkern5mu$\hidewidth\cr$\m@th#1\to$\cr}%
}
\makeatother

\newcommand{\xmrightarrow}[1]{\overset{#1}{\pgets}}

\begin{document}

\title{Cauchy completeness and causal spaces}
\author{Branko Nikoli\'c*}
\address{Mathematics department, Macquarie University, Sydney, Australia}
\email{branko.nikolic@mq.edu.au}
\thanks{*I gratefully acknowledge the support of an International Macquarie University Research Scholarship.}

\date{\today}

\subjclass[2010]{18D20,16D90}

\keywords{Cauchy completeness, proper time, causal spaces}

\begin{abstract}
	Following Lawvere's description of metric spaces using enriched category theory, we introduce a change in the base of enrichment that allows description of some aspects of (relativistic) causal spaces. All such spaces are Cauchy complete, in the sense of enriched category theory. Furthermore, we give sufficient conditions on a base monoidal category for which enriched categories are Cauchy complete if and only if their underlying categories are (their idempotent arrows split).
\end{abstract}

\maketitle

\tableofcontents

\input{Cauchy_01_main}

\appendix

\input{Cauchy_02_appendix}

\bibliographystyle{hacm}
\bibliography{references.bib}

\end{document}

%% file: Cauchy_01_main.tex
\section{Introduction}

A generalized metric space $X$ consists of a set of points and, for each pair of points $P$ and $Q$, a distance $d(P,Q)\in[0,\infty]$ from $P$ to $Q$ such that, for all points $P$, $Q$ and $R$,
\begin{align} \label{ineq:0}
d(P,P)&=0\\
d(P,Q)+d(Q,R)&\geq d(P,R)\,. \label{ineq:triangle}
\end{align}
``Generalized'' comes from dropping conditions of finiteness (allowing infinite distance), symmetry (allowing $d(P,Q)\neq d(Q,P)$), and distinguishability (allowing $d(P,Q)=0$ without $P=Q$). Those spaces correspond precisely \cite{Lawvere1973} to categories enriched in $\mathcal{R}$ - a monoidal category (more concretely, a totally ordered set) with positive reals and infinity as objects, an arrow between $a$ and $b$ if and only if $b\leq a$, and monoidal structure given by sum. $\mathcal{R}$ is also closed, with internal hom given by truncated subtraction, uniquely defined right adjoint to summation. To see the correspondence, recall \cite{Kelly1982} that a category $\mathcal{X}$ enriched in a monoidal category $\mathcal{V}$ consists of a set of objects (points in this case), for each pair of objects a hom, that is, an object in $\mathcal{V}$ (a number providing distance in this case), and unit and composition arrows of $\mathcal{V}$ (providing (in)equalities (\ref{ineq:0})-(\ref{ineq:triangle}), in this case) satisfying unit and associativity laws (trivially true in this case because $\mathcal{R}$ is a poset).

Denote by $\mathcal{I}$ the space having only one point $*$. An enriched module (aka profunctor, distributor) $\mathcal{I}\xmrightarrow{M}\mathcal{X}$, alternatively expressed as an enriched presheaf $M:\mathcal{X}^\mathrm{op}\rightarrow \mathcal{R}$, assigns to each point $P$ in $\mathcal{X}$ a distance from $P$ to $*$, $M(P,*)$, with an action ensuring triangle inequality for the newly introduced distances
\begin{equation}\label{eq:metM}
\mathcal{X}(P,Q)+M(Q,*)\geq M(P,*)\,.
\end{equation}
For example, each point $P\in \mathcal{X}$ defines a module $M_P(Q,*)=\mathcal{X}(Q,P)$ - this motivates a general definition \ref{def:Conv} for convergent modules. Dually, an enriched module $\mathcal{X}\xmrightarrow{N}\mathcal{I}$ assigns to each point $P$ in $\mathcal{X}$ a distance from $*$ to $P$, with actions
\begin{equation}\label{eq:metN}
N(*,P)+\mathcal{X}(P,Q)\geq N(*,Q)\,.
\end{equation}
Asking for $M$ and $N$ to form an adjunction in $\mathcal{R}\text{-}\mathrm{Mod}$ imposes existence of a counit
\begin{equation}
M(P,*)+N(*,Q)\geq \mathcal{X}(P,Q)
\end{equation}
expressing that the newly introduced distances do not violate the triangular inequality via $*$, enabling us to consider a new space $\mathcal{X}_*$, with an added point $*$. Finally, the unit of the adjunction\footnote{The coend involved in the module composition reduces to $\mathrm{inf}$ when the base of enrichment is $\mathcal{R}$.}
\begin{equation}
0\geq \underset{P\in\mathcal{X}}{\mathrm{inf}} N(*,P)+M(P,*)
\end{equation}
forces the newly adjoined point to have zero distance from (and to) the rest of the space, providing a Cauchy condition analogous to the one for Cauchy sequences. This motivates general definitions \ref{def:Cauchy} of Cauchy modules, and of Cauchy completeness of enriched categories \ref{def:ccCat}. Another important base is the monoidal category $\mathrm{Ab}$ of Abelian groups, where one-object $\mathrm{Ab}$-categories are rings, and they are Morita equivalent (have equivalent categories of (left) modules) if and only if their Cauchy completions are equivalent \cite{Kelly1982}. We review definitions and some results related to general Cauchy completeness in appendix \ref{sec:cc}.

In section \ref{sec:es} we give a modification of the base category $\mathcal{R}$, call it $\mathcal{R}_\bot$, which gives causal spaces as $\mathcal{R}_\bot$-enriched categories, and explain how black holes and wormholes can be described using enriched modules. We also prove a surprising fact that all causal spaces are Cauchy complete, in the sense of enriched category theory.

In section \ref{sec:easy} we give conditions on a monoidal category $\mathcal{V}$ which ensure that a $\mathcal{V}$-category $\mathcal{C}$ is Cauchy complete if and only if the underlying ($\mathrm{Set}$-enriched) category $\mathcal{C}_0$ is Cauchy complete, which for $\mathrm{Set}$-enrichment means that idempotents in $\mathcal{C}_0$ split. As a corollary we add a few more conditions on $\mathcal{V}$ ensuring that all $\mathcal{V}$-enriched categories are Cauchy complete, generalizing the case of $\mathcal{R}_\bot$.

\section{Causal spaces as enriched categories}\label{sec:es}

Given a space-time $E$ one can assign to each time-like path $p$ in $E$ its proper time $T(p)$. Maximizing the proper time $T(p)$ over all time-like paths between two events gives an interval or ``distance'' between them.
This is not distance in the sense of a metric space, mainly because the triangle inequality is inverted. The maximal time will usually (in physical situations) correspond to time measured by an inertial observer, while any accelerated reference frame would measure a shorter time, with a photon bouncing from appropriately set up mirrors would ``measure'' a zero time. However, we used maximizing over all time-like paths, rather than an inertial path, because of possible existence of Lorentzian manifolds where there are causally related points which do not have a (unique) inertial path between them. This is analogous to minimizing path length over all paths on a Riemannian manifold to obtain metric, where, for example antipodal points on a sphere have multiple shortest paths, or two points in a plane on the opposite side of a cut out (closed) disc have no path with a minimal length between them.

To get the inverted triangular inequality one could just invert the arrows of $\mathcal{R}$. On one hand, such a category could no longer be closed because the object $0$ would be the monoidal identity and the initial object at the same time, which would mean that tensoring (summing) does not preserve colimits (in particular, the initial object), since, for example
\begin{equation}
1=1+0\neq 0\,.
\end{equation}
On the other hand, physically, there would be no object in the monoidal category that could be assigned to space-like separated events. Both of the problems are solved by freely adding an initial object which we denote by $\bot$. So, the correct base for enrichment is formally given by
\begin{defn}
A symmetric closed monoidal category 
$\mathcal{R_\bot }$ is defined to have
\begin{itemize}
\item objects the real positive numbers $[0,\infty)$ with infinity $\infty$ and and the additional object $\bot$
\item arrows $a\rightarrow b$ existing uniquely if $a=\bot$, $b=\infty$ or $a\leq b$, forming a total order
\item tensor product $+:\mathcal{R_\bot }\times\mathcal{R_\bot }\rightarrow \mathcal{R_\bot }$ given by
\begin{equation}
\begin{tabular}{c|c|c|c}
+&$\bot$&$b$&$\infty$\\\hline
$\bot$&$\bot$&$\bot$&$\bot$\\\hline
$a$&$\bot$&$a+b$&$\infty$\\\hline
$\infty$&$\bot$&$\infty$&$\infty$\\
\end{tabular}
\end{equation}
\item internal hom $-:\mathcal{R_\bot }^\mathrm{op}\times\mathcal{R_\bot }\rightarrow \mathcal{R_\bot }$ given by
\begin{equation}
\begin{tabular}{c|c|c|c}
-&$\bot$&$b$&$\infty$\\\hline
$\bot$&$\infty$&$\infty$&$\infty$\\\hline
$a$&$\bot$&$
\begin{cases} b-a, & a\leq b \\
\bot, & a>b \end{cases}
$&$\infty$\\\hline
$\infty$&$\bot$&$\bot$&$\infty$\\
\end{tabular}
\end{equation}
\end{itemize}
\end{defn}
With this direction of arrows, all the colimits are suprema, and limits are infima.

A category $\mathcal{E}$ enriched in $\mathcal{R_\bot }$ has objects $X,Y,...$ interpreted as events, and homs $\mathcal{E}(X,Y)\in\mathcal{R_\bot }$ interpreted as ``distances'' or intervals. If $\mathcal{E}(X,Y)=\bot$ then $Y$ is not in the future of $X$, equivalently said, $X$ cannot cause $Y$. The composition of homs witnesses that the chosen time between the two events is the largest,
\begin{equation}\label{eq:compR}
\mathcal{E}(X,Y)+\mathcal{E}(Y,Z) \leq \mathcal{E}(X,Z) 
\end{equation}
and the unit
\begin{equation}\label{eq:unitR}
0\leq \mathcal{E}(X,X)
\end{equation}
 prevents endohoms from being $\bot$. The associativity and unit axioms are trivially satisfied because $\mathcal{R_\bot }$ is a poset.

\begin{example}
In a Minkowski 2D space-time objects are points in $(t,x)\in\mathbb{R}^2$ and homs are
\begin{equation}
\mathcal{E}((t,x),(t',x')) = \begin{cases} \sqrt{(t'-t)^2-(x'-x)^2}, &\mbox{if } t'-t\geq |x'-x| \\
\bot, & \mathrm{otherwise} \end{cases}
\end{equation}
\end{example}

\begin{prop}\label{prop:homs}
Properties of homs of $\mathcal{E}$ include
\begin{enumerate}
\item endohoms are monoidal idempotents
\begin{equation}\label{eq:endoE}
\mathcal{E}(X,X)+\mathcal{E}(X,X)=\mathcal{E}(X,X)
\end{equation}
\item the action of endohoms on other homs is given by equalities
\begin{align}
\mathcal{E}(Y,X)+\mathcal{E}(X,X)=\mathcal{E}(Y,X)\label{eq:actionEq1}\\
\mathcal{E}(X,X)+\mathcal{E}(X,Y)=\mathcal{E}(X,Y)\label{eq:actionEq2}
\end{align}
\item possible endohoms are
\begin{equation}
\mathcal{E}(X,X)=0\;\;\mathrm{or}\;\;\mathcal{E}(X,X)=\infty
\end{equation}
\begin{enumerate}
\item
if $\mathcal{E}(X,X)=\infty$, all the homs $\mathcal{E}(Y,X)$ and $\mathcal{E}(X,Y)$ are either $\bot$ or $\infty$
\item
if $\mathcal{E}(X,X)=0$, either both $\mathcal{E}(X,Y)$ and $\mathcal{E}(Y,X)$ equal $0$ or at least one equals $\bot$
\end{enumerate}
\end{enumerate}
\end{prop} 
\begin{proof}
\begin{enumerate}
\item Adding $\mathcal{E}(X,X)$ to the unit (\ref{eq:unitR}) gives
\begin{equation}
\mathcal{E}(X,X)\leq\mathcal{E}(X,X)+\mathcal{E}(X,X)  
\end{equation}
On the other hand, the composition (\ref{eq:compR}) for $Y$ and $Z$ equal $X$ gives
\begin{equation}
\mathcal{E}(X,X)+\mathcal{E}(X,X)\leq \mathcal{E}(X,X) 
\end{equation}
\item Adding $\mathcal{E}(X,Y)$ to the unit, and and the compositions
\begin{align}
\mathcal{E}(X,Y)+\mathcal{E}(X,X)\leq \mathcal{E}(X,Y) \\
\mathcal{E}(X,X)+\mathcal{E}(Y,X)\leq \mathcal{E}(Y,X)
\end{align}
give the required result.
\item By part (1) of the proposition, noting that objects $\bot$, $0$ and $\infty$ are the only monoidal idempotents in $\mathcal{R}_\bot$, and using the unit (\ref{eq:unitR}), restricts possible endohoms to $0$ and $\infty$.
\begin{enumerate}
\item Case analysis on (\ref{eq:actionEq1})-(\ref{eq:actionEq2})
\item Case analysis on $\mathcal{E}(Y,X)+\mathcal{E}(X,Y)\leq \mathcal{E}(X,X)=0$
\end{enumerate}
\end{enumerate}
\end{proof}
\begin{remark}
 Call an event with the infinite endohom (situation (3a)) \textit{ irregular}. Although unphysical, these are needed to keep $\mathcal{R}_\bot$ closed. For instance, $\bot$ is irregular, since $\bot-\bot=\infty$. However, part (3a) of Proposition \ref{prop:homs} ensures that such points in space are either causally unrelated to, or at an infinite temporal distance from, the rest of the (physical) space.
Part (3b) of Proposition \ref{prop:homs} prevents the grandfather paradox in the physical part of the space - given two regular (endohom being $0$) events $X$ and $Y$, it is not possible for both of them to cause each other, unless they happen simultaneously.
\end{remark}

\begin{remark}
A program for formulating quantum gravity using discrete partial orders, started in \cite{Bombelli1987} and reviewed in, for example, \cite{Dowker2013}, has a notion of \textit{causal set} as a basic mathematical structure. If we take the underlying category $\mathcal{E}_0$ of a causal space $\mathcal{E}$, we get a general preordered set without requirements for antisymmetry and local finiteness - the information about local time-like intervals is contained in homs, and allows different events to happen at the same point in space-time. On the other hand, each causal set has a corresponding causal space, where homs come from the local finiteness condition - if $A$ causes $B$, then $\mathcal{E}(A,B)$ is the (integer) length of the longest (necessarily finite) path between $A$ and $B$.
\end{remark}

\subsection{Enrichment in $[-\infty,\infty]$}

A possible generalization of both metric and event spaces, would be enrichment in $[-\infty,\infty]$, with an arrows from $A$ to $B$, if $B\leq A$. Then positive length would denote space-like intervals, with triangle inequality (\ref{ineq:triangle}), while negative numbers would be interpreted as time-like intervals. However, the triangle inequality with mixed entries is too restrictive, so the Minkowski 2D space-time is not enriched in $[-\infty,\infty]$. For example,
\begin{align}
A&=(0,0)\\
B&=(-1,0)\\
C&=(0,1)
\end{align}
gives
\begin{align}
\mathcal{E}(A,B)+\mathcal{E}(B,C)&= -1+0=-1\\
\mathcal{E}(A,C)&=1\,.
\end{align}

\subsection{$\mathcal{R}_\bot\text{-}\mathrm{Cat}$}

An $\mathcal{R_\bot }$-functor $F:\mathcal{D}\rightarrow\mathcal{E}$ maps events in $\mathcal{D}$ to events in $\mathcal{E}$ such that the distances increase
\begin{equation}\label{eq:opLawFun}
\mathcal{D}(A,B)\leq \mathcal{E}(FA,FB)\,.
\end{equation}
In particular, space-like intervals (given by $\bot$) can map to time-like intervals.

Natural transformations $\eta:F\rightarrow G$ indicate that for all $A\in D$ the event $GA$ is in the future of $FA$.

Since $\mathcal{R_\bot }$ is symmetric, closed and (co)complete, so is $\mathcal{R_\bot }\text{-}\mathrm{Cat}$ \cite{Kelly1982}. Explicitly, the tensor product $\mathcal{D}+\mathcal{E}$ of $\mathcal{D}$ and $\mathcal{E}$ has
\begin{itemize}
\item objects pairs $(A,X)$
\item homs $(\mathcal{D}+\mathcal{E})((A,X),(B,Y))=\mathcal{D}(A,B)+\mathcal{E}(X,Y)$
\end{itemize}
and $[\mathcal{D},\mathcal{E}]$ has
\begin{itemize}
\item objects $\mathcal{R_\bot }$-functors $F$, $G$...
\item homs
\begin{equation}
 [\mathcal{D},\mathcal{E}](F,G)=\int_{A\in\mathcal{D}}\mathcal{E}(FA,GA)
=\inf_{A\in\mathcal{D}}\mathcal{E}(FA,GA)\,.
\end{equation}
\end{itemize}

Finally, given a causal space $\mathcal{E}$, using symmetry of $\mathcal{R}_\bot$ we can form the opposite $\mathcal{E}^\mathrm{op}$ by taking the same set of objects and 
\begin{equation}
\mathcal{E}^\mathrm{op}(X,Y)=\mathcal{E}(Y,X)
\end{equation}
for homs.

\subsection{Modules, black holes and wormholes}

A (2-sided) module $M:\mathcal{D}\nrightarrow\mathcal{E}$ is defined as an $\mathcal{R}_\bot$-functor
\begin{equation}
M:\mathcal{E}^\mathrm{op}+\mathcal{D}\rightarrow\mathcal{R}_\bot
\end{equation}
and can be equivalently given by actions
\begin{align}
\mathcal{E}(Y,X)+M(X,A)\leq M(Y,A)\label{eq:mod}\\
M(X,A)+\mathcal{D}(A,B)\leq M(X,B)\label{eq:mod2}\,.
\end{align}
These inequalities enable us to ``glue'' the two causal spaces with homs between objects of $\mathcal{E}$ and $\mathcal{D}$ given by $M$, and all homs from $\mathcal{D}$ to $\mathcal{E}$ being $\bot$, a process known as a lax colimit or collage \cite{Street2001}.
\begin{remark}
Physically, such a module can be interpreted as a wormhole going from $\mathcal{E}$ to $\mathcal{D}$. In particular, when $\mathcal{D}=\mathcal{I}$ the module $M$ is a black hole in $\mathcal{E}$.
\end{remark}

Composition of modules $N:\mathcal{C}\nrightarrow\mathcal{D}$ and $M:\mathcal{D}\nrightarrow\mathcal{E}$ is given by
\begin{align}
(M\circ N)(X,P)&= \int^{A\in\mathcal{D}} M(X,A)+N(A,P)\\
&=  \sup_{A\in\mathcal{D}}M(X,A)+N(A,P)
\end{align}
for all $P\in \mathcal{C}$ and $X\in \mathcal{E}$.

\subsection{Cauchy completeness}

To give a pair of adjoined modules 
$(M\dashv N):\mathcal{I}\xmrightarrow{}\mathcal{E}$
is the same as to give a pair of $\mathcal{R_\bot }$-functors
\begin{align}
M&:\mathcal{E}^\mathrm{op}\rightarrow\mathcal{R_\bot }\\
N&:\mathcal{E}\rightarrow\mathcal{R_\bot }
\end{align}
which, in addition to the actions (\ref{eq:mod})-(\ref{eq:mod2})
\begin{align}
\mathcal{E}(Y,X)+M(X)\leq M(Y)\label{eq:ineqRCauchy}\\
N(X)+\mathcal{E}(X,Y)\leq N(Y)\label{eq:ineqLCauchy}
\end{align}
satisfy (existence of the unit and counit of the adjunction)
\begin{align}
0 &\leq \sup_{X} (N(X)+M(X))\label{eq:Zpick}\\
\mathcal{E}(X,Y)&\geq M(X)+N(Y)\,.\label{eq:wormholeCounit}
\end{align}

\begin{prop}
Any $\mathcal{R_\bot}$ enriched category $\mathcal{E}$ is Cauchy complete.
\end{prop}
\begin{proof}
First, consider the case when $\mathcal{E}$ is empty. Then $M$ and $N$ are unique empty functors, but they cannot be adjoint as the RHS of (\ref{eq:Zpick}) equals $\bot$
Since there are no Cauchy modules, $\mathcal{E}$ is Cauchy complete.

Now, assume $\mathcal{E}$ is non-empty and $M$ is a Cauchy module, that is there is $N$ such that (\ref{eq:ineqRCauchy})-(\ref{eq:wormholeCounit}) hold. In particular, since $\bot$ is the only element smaller than $0$, equation (\ref{eq:Zpick}) implies that there is $Z\in \mathcal{E}$ such that
\begin{equation}\label{eq:eqZ0}
0\leq N(Z)+M(Z)\,.
\end{equation}
If either $N(Z)$ or $M(Z)$ was equal to $\bot$ the sum would equal $\bot$ as well, so we have that both terms are greater or equal than $0$,
\begin{equation}\label{eq:eqZ}
0\leq N(Z) \;\;\mathrm{and}\;\; 0\leq M(Z)\,.
\end{equation}

Now we have 
\begin{align}
M(Y)&\leq M(Y)+N(Z)\\
&\leq \mathcal{E}(Y,Z)\\
&\leq \mathcal{E}(Y,Z)+M(Z)\\
&\leq M(Y)
\end{align}
proving that $M(Y)=\mathcal{E}(Y,Z)$, and showing that $Z$ represents $M$.
\end{proof}

\section{Cauchy completeness via idempotent splitting}\label{sec:easy}

Here we consider which monoidal categories $\mathcal{V}$ produce enriched categories whose Cauchy completeness is determined by idempotent splitting in the corresponding underlying category. We begin with an easy direction.  
\begin{prop}
Let $\mathcal{V}$ be a locally small, cocomplete symmetric monoidal closed category. If a small $\mathcal{V}$-category $\mathcal{E}$ is Cauchy complete then idempotents split in the underlying category $\mathcal{E}_0$.
\end{prop}
\begin{proof}
Let $I\xrightarrow{e}\mathcal{E}(E,E)$ be an idempotent in $\mathcal{E}_0$. Let $E_*:\mathcal{I}\nrightarrow\mathcal{E}$ and $E^*:\mathcal{E}\nrightarrow\mathcal{I}$ denote the modules induced by the $\mathcal{V}$-functor picking the object $E$. That is
\begin{align}
E_*(X)=\mathcal{E}(X,E)\\
E^*(X)=\mathcal{E}(E,X)
\end{align}
with actions given by composition in $\mathcal{E}$.
The induced module endomorphisms  $e_*:E_*\Rightarrow E_*$ and $e^*:E^*\Rightarrow E^*$ are idempotent because $e$ is.
Since in the corresponding presheaf category idempotents split, 
there is a module $M:\mathcal{I}\nrightarrow\mathcal{E}$, and module morphisms $f:E_*\Rightarrow M$, $g:M\Rightarrow E_*$ splitting $e_*$. Similarly, there is a module $N:\mathcal{E}\nrightarrow\mathcal{I}$, and module morphisms $k:E^*\Rightarrow N$, $l:N\Rightarrow E^*$ splitting $e^*$. Using the fact that $e^*$ and $e_*$ are mates under the adjunction $E_*\dashv E^*$, it is easy to show that\footnote{Here $\otimes$ denotes the horizontal composition, and $\circ$ the vertical composition of module morphisms.}
 $(k\otimes f)\circ\eta$ and $\epsilon \circ(g \otimes l)$ are unit and a counit of the adjunction $M\dashv N$. Since $\mathcal{E}$ is Cauchy complete, $M$ is represented by an object, say $D\in \mathcal{E}$, and so, using the weak Yoneda lemma, $e$ splits through it. 
\end{proof}

\begin{prop}\label{prop:CC0CC}
Consider the following properties of a cocomplete, locally small, symmetric monoidal closed category $\mathcal{V}$:
\begin{enumerate}[label=(\roman*)]
\item the underlying functor
\begin{equation}
\mathcal{V}(I,-):\mathcal{V}\rightarrow \mathrm{\bf Set}
\end{equation}
takes regular epi families to epi families (joint surjections),
\item the function
\begin{equation}
\mathcal{V}(I,A)\times\mathcal{V}(I,B)
\xrightarrow{\otimes}\mathcal{V}(I\otimes I,A \otimes B)
\xrightarrow{-\circ\rho_I}\mathcal{V}(I,A \otimes B)
\end{equation}
is a bijection,
\end{enumerate}
then a small $\mathcal{V}$-category $\mathcal{E}$ is Cauchy complete if idempotents split in the underlying category $\mathcal{E}_0$.
\end{prop}
\begin{proof}
Let $M:\mathcal{I}\xmrightarrow{}\mathcal{E}$ be a Cauchy module with a right adjoint $N$ which amounts to giving actions
\begin{align}
M(X) \otimes \mathcal{E}(Y,X)\xrightarrow{\alpha_{Y,X}} M(Y)\label{eq:VineqRCauchy2}\\
\mathcal{E}(X,Y) \otimes N(X)\xrightarrow{\beta_{X,Y}} N(Y)\label{eq:VineqLCauchy2}
\end{align}
compatible with unit and composition in $\mathcal{E}$,
and unit and counit for the adjunction
\begin{align}
\eta&:I\xrightarrow{}\int^{Y}M(Y)\otimes N(Y)\, ,\\
\epsilon_{X,Y}&:N(Y)\otimes M(X)\rightarrow \mathcal{E}(X,Y)\, .
\end{align}
The coend cowedge components
\begin{equation}
M(X)\otimes N(X)\xrightarrow{w_X}\int^{Y}M(Y)\otimes N(Y)
\end{equation}
form a jointly regular epic family, see section \ref{sec:epi} example \ref{ex:moduleReg}. By condition (i), the functor $\mathcal{V}(I,-)$ takes them to a jointly surjective family of functions $\mathcal{V}(I,w_X)$. This in particular means that the unit of the adjunction is in the image of a function $\mathcal{V}(I,w_Z)$, for some $Z$. So, the unit decomposes as $\eta=w_Z\circ{z}$.
From condition (ii)  we get that ${z}$ can be further decomposed as  ${m}\otimes {n}$ for a unique pair of maps ${m}:I\rightarrow M(Z)$ and ${n}:I\rightarrow N(Z)$, to give a final decomposition of the unit
\begin{equation}\label{eqn:splitting}
\eta=w_Z\circ({m}\otimes {n})
\end{equation}

One of the adjunction axioms, together with (\ref{eqn:splitting}) gives a commutative diagram shown in (\ref{diag:adj}).
\begin{equation}\label{diag:adj}
\begin{tikzpicture}[baseline=(current  bounding  box.center)]
\def \strx {3.0};
\def \stry {1.5};

\node (Gu2) at (0 * \strx ,2 * \stry) {$M(Y)$};
\node (Gu1) at (0 * \strx ,1 * \stry) {$I\otimes M(Y)$};
\node (G) at (0 * \strx ,0 * \stry) {$\int^{C}M(C)\otimes N(C)\otimes M(Y)$};
\node (Gd1) at (0 * \strx ,-1 * \stry) {$\int^{C}M(C)\otimes \mathcal{E}(Y,C)$};
\node (Gd2) at (0 * \strx ,-2* \stry) {$M(Y)$};

\node (Gr1) at (2 * \strx ,0 * \stry) {$M(Z)\otimes N(Z)\otimes M(Y)$};
\node (Gr2) at (2 * \strx ,-1 * \stry) {$M(Z)\otimes \mathcal{E}(Y,Z)$};

\path[->,>=angle 90]
(Gu2) edge [bend right=100] node[left] {$1$} (Gd2);
\path[->,>=angle 90]

(Gu2) edge node[left] {$\cong$} (Gu1);
\path[->,>=angle 90]
(Gu1) edge node [left] {$\eta\otimes 1$} (G);
\path[->,>=angle 90]
(G) edge node [left] {$\int^{C}1\otimes \epsilon_{Y,C}$} (Gd1);
\path[->,>=angle 90] 
(Gd1) edge  node[left] {$\cong$} (Gd2);

\path[->,>=angle 90]
(Gu1) edge node [above right] {${m}\otimes{n}\otimes 1$} (Gr1);
\path[->,>=angle 90]
(Gr1) edge node [above] {$w_Z\otimes 1$} (G);
\path[->,>=angle 90]
(Gr1) edge node [right] {$1\otimes \epsilon_{Y,Z}$} (Gr2);
\path[->,>=angle 90]
(Gr2) edge node [above] {$w_Z$} (Gd1);

\end{tikzpicture}
\end{equation}

From the outside of the diagram (\ref{diag:adj}) it follows that the identity on $M(Y)$ decomposes into the following two maps
\begin{align}
	M(Y)\xrightarrow{{n}\otimes 1}N(Z)\otimes M(Y)\xrightarrow{\epsilon_{Y,Z}} \mathcal{E}(Y,Z)\\
	\mathcal{E}(Y,Z)\xrightarrow{{m}\otimes 1}M(Z)\otimes \mathcal{E}(Y,Z)\xrightarrow{\alpha_{Y,Z}}M(Y)\,.
\end{align}
Both of these sets of arrows are $\mathcal{V}$-natural in $Y$, following from $\mathcal{V}$-naturality of $\epsilon$ and compatibility of action $\alpha$ with composition in $\mathcal{E}$. Composing them the other way around we get an idempotent $\mathcal{V}$-natural transformation on $\mathcal{E}(-,Z)$, which is represented by an idempotent arrow $Z\xrightarrow{e}Z$ in $\mathcal{E}_0$. Since idempotents split, there is $Z'$ through which $e$ splits, hence $Z'$ is a representing object for $M$.
\end{proof}

\begin{remark}
The only place we used symmetry and closedness of $\mathcal{V}$ was the definition of module compositions using coends, and the definition of the category of enriched presheaves. Both of these notions are definable for non-symmetric $\mathcal{V}$, or even when the base of enrichment is a bicategory \cite{Street2005}, so we expect the above theorems to work at that level of generality as well.
\end{remark}

\begin{cor}
A cocomplete quantale $\mathcal{Q}$ such that any collection of its objects $\{A_i\}$ with an arrow
\begin{equation}
I\rightarrow \bigvee_i A_i
\end{equation}
contains an object $Z\in \{A_i\}$ with an arrow
\begin{equation}
I\rightarrow Z
\end{equation}
has the property that all small $\mathcal{Q}$-categories $\mathcal{E}$ are Cauchy complete.
\end{cor}
\begin{example}
The motivating example $\mathcal{R}_\bot$ has this property.
\end{example}
\begin{cor}
If a cocomplete category $\mathcal{V}$ is Cartesian closed and
\begin{equation}
\mathcal{V}(1,-):\mathcal{V}\rightarrow \mathrm{Set}
\end{equation}
has a right adjoint, then $\mathcal{V}$ satisfies the requirements of proposition \ref{prop:CC0CC}. 
\end{cor}
Denoting by $G$ the right adjoint we need a (natural) bijection
\begin{equation}\label{eq:PshIso}
 \mathcal{V}(A,GS)\cong\mathrm{\bf Set}(\mathcal{V}(I,A),S)\,.
\end{equation}
\begin{example}
For $\mathcal{V}=\mathrm{Set}$, $G=1_\mathrm{Set}$. More generally, if 
$\mathcal{V}=[\mathcal{C}^\mathrm{op},\mathrm{Set}]$ and $\mathcal{C}$ has a terminal object $1$ then
\begin{align}
(GS)C=\mathrm{Set}(\mathcal{C}(1,C),S)
\end{align}
functorially in $C$. The isomorphism (\ref{eq:PshIso}) follows from
\begin{align}
[\mathcal{C}^\mathrm{op},\mathrm{Set}]&(A,\mathrm{Set}(\mathcal{C}(1,-),S))\\
&\cong \int_{C\in\mathcal{C}} \mathrm{Set}(AC,\mathrm{Set}(\mathcal{C}(1,C),S))\\
&\cong \mathrm{Set}\left( \int^{C\in\mathcal{C}} AC\times \mathcal{C}(1,C),S\right)\\
&\cong \mathrm{Set}(A1,S)\\
&\cong \mathrm{Set}\left( [\mathcal{C}^\mathrm{op},\mathrm{Set}](\mathcal{C}(-,1),A),S\right)\\
&\cong \mathrm{Set}\left( [\mathcal{C}^\mathrm{op},\mathrm{Set}](1,A),S\right)
\end{align}
where $1$ in the last line denotes the terminal presheaf which is the monoidal unit in $[\mathcal{C}^\mathrm{op},\mathrm{Set}]$.
\end{example}
\begin{example}
For $\mathcal{V}=\mathrm{Cat}$, $GS$ is the chaotic category on the set $S$, because mapping into it is uniquely determined by the assignment on objects. More generally, for $\mathcal{V}=n\text{-}\mathrm{Cat}$, $GS$ is a the chaotic category seen as a locally discrete $n$-category (each hom is the terminal $(n-1)$-category).
\end{example}

In some cases condition (2) holds when the product is not Cartesian.
\begin{example}
$\mathrm{Gray}_{(l)}$ has the same objects and arrows as $2\text{-}\mathrm{Cat}$, but (lax) Gray tensor product, rather than the Cartesian one for the monoidal structure. Strict functors $1\rightarrow \mathcal{A}\otimes_{(l)}\mathcal{B}$ detect (pick) objects, which are pairs consisting of an object in $\mathcal{A}$ and an object in $\mathcal{B}$, hence satisfying condition (ii). 
\end{example}

\begin{prop} \label{prop:allCC}
Let $\mathcal{V}$ be a monoidal category. The following are equivalent:
\begin{enumerate}
	\item every $\mathcal{V}$-category $\mathcal{C}$ has a Cauchy complete underlying category $\mathcal{C}_0$,
	\item every monoid $(T,\mu,\eta)$ in $\mathcal{V}$ induces an idempotent-splitting monoid on the hom-set $\mathcal{V}(I,T)$.
\end{enumerate}
\end{prop}
\begin{proof}
	$(1\Rightarrow2)$ Consider a one-object category $\mathcal{C}$ with the endohom, multiplication and unit given by $(T,\mu,\eta)$. The underlying category is precisely the suspension of the monoid $\mathcal{V}(I,T)$, so idempotent-splitting in $\mathcal{C}_0$ is the same as idempotent-splitting in $\mathcal{V}(I,T)$.\newline
	$(2\Rightarrow1)$ Let $I\xrightarrow{e} \mathcal{C}(A,A)$ be an idempotent in $\mathcal{C}_0$. Since $\mathcal{C}(A,A)$ is a monoid in $\mathcal{V}$, $e$ is also an idempotent in the induced monoid on $\mathcal{V}(I,\mathcal{C}(A,A))$, and, by condition 2, it splits. 
\end{proof}

\begin{remark}
	Under condition 2, all idempotents in $\mathcal{C}_0$ split through the same object they live on. As a consequence, if an array of maps composes to the identity on an object $A$, then all intermediate objects are isomorphic to $A$.
\end{remark}
\begin{cor}
A monoidal category $\mathcal{V}$ satisfying conditions of the proposition \ref{prop:CC0CC}, and the second of \ref{prop:allCC}, has all small $\mathcal{V}$-categories Cauchy complete.
\end{cor}

%% file: Cauchy_02_appendix.tex
\section{Cauchy completeness}\label{sec:cc}

Here we summarize basic definitions and results related to the general theory of Cauchy completeness. The motivating example is in the introduction.

\begin{defn}\label{def:Cauchy}
A $\mathcal{V}$-module $M:\mathcal{B}\xmrightarrow{}\mathcal{C}$ is called Cauchy if it has a right adjoint in $\mathcal{V}\text{-}\mathrm{Mod}.$ 
\end{defn}

\begin{prop}\cite{Street1983}
A $\mathcal{V}$-module $M$ is Cauchy if and only if all $M$-weighted colimits are absolute.
\end{prop}
More on absolute colimits in $(\mathrm{Set}\text{-})$categories can be found in \cite{Pare1971}. Absolute weights for enrichment in a bicategory were further examined in \cite{Garner2014}.

\begin{prop}\cite{Kelly1982} For symmetric closed complete and cocomplete $\mathcal{V}$,
a $\mathcal{V}$-module $M:\mathcal{I}\xmrightarrow{}\mathcal{C}$ is Cauchy if and only if it is small-projective, that is, the representable functor
\begin{equation}
[\mathcal{C}^\mathrm{op},\mathcal{V}](M,-):[\mathcal{C}^\mathrm{op},\mathcal{V}]\rightarrow \mathcal{V}
\end{equation}
preserves small colimits.
\end{prop}

\begin{defn}\label{def:Conv}
A right $\mathcal{C}$-module $M:\mathcal{B}\xmrightarrow{}\mathcal{C}$ is called convergent if there is a $\mathcal{V}$-functor $F:\mathcal{B}\rightarrow \mathcal{C}$ such that $M\cong F_*:= \mathcal{C}(-,F-)$.
\end{defn}
When $\mathcal{B}=\mathcal{I}$, $M$ being convergent is equivalent to $M$ being representable in the usual sense.

\begin{defn}\label{def:ccCat}
A $\mathcal{V}$-category $\mathcal{C}$ is Cauchy complete if all Cauchy modules into $\mathcal{C}$ are representable.
\end{defn}

\begin{prop}
A $\mathcal{V}$-category $\mathcal{C}$ is Cauchy complete if and only if it has all absolute-weighted colimits.
\end{prop}

\section{Familial epiness}\label{sec:epi}

In this section we explore the notion of jointly epi families and how it can be extended to extremal, strong and regular epi families. The letter  $\mathcal{V}$ denotes an ordinary category. Most of the concepts here are taken from \cite{Street1984}.

\begin{defn}
A family of maps $\{A_i\xrightarrow{w_i}B\}_{i\in I}$
in $\mathcal{V}$ is {\it jointly epi} if any two maps $B\xrightarrow{f}C$ and $B\xrightarrow{g}C$ satisfying, for all $i$,
$
f\circ w_i=g\circ w_i
$
implies $f=g$.
\end{defn}

\begin{defn}
A family of maps $\{A_i\xrightarrow{w_i}B\}_{i\in I}$ 
in $\mathcal{V}$ is {\it jointly extremal epi} if it is jointly epi and satisfies the {\it invertible mono} condition - that any mono $m$ through which all $w_i$ factor is necessarily an isomorphism.
\end{defn}

\begin{defn}
A family of maps $\{A_i\xrightarrow{w_i}B\}_{i\in I}$
in $\mathcal{V}$ is {\it jointly strong epi} if it is jointly epi and satisfies the {\it diagonal fill in} condition - that any map $B\xrightarrow{g}D$, any mono $C\xrightarrow{m}D$, and any family of maps $\{A_i\xrightarrow{f_i}C\}_{i\in I}$ such that $m\circ f_i=g\circ w_i$, there is a unique diagonal filler $B\xrightarrow{d}C$ such that all triangles commute.
\end{defn}
\begin{remark}
As in the single epi case, if equalizers exist in $\mathcal{V}$, the condition of being jointly epi in order to be jointly extremal/strong, follows from the invertible-mono/diagonal-fill-in condition.
\end{remark}
\begin{remark}
As in the single epi case, any jointly strong epi family is jointly extremal epi, and in the presence of pullbacks, every jointly extremal epi family is a jointly strong epi family.
\end{remark}

\begin{defn}
A {\it relation} $R$ on a family $\{A_i\}_{i\in I}$ of objects in $\mathcal{V}$ is given by a set $R_{i,j}$ of spans between $A_i$ and $A_j$, for each $i$ and $j$. We use $R$ to denote the (disjoint) union of all $R_{i,j}$. A {\it quotient} of $R$ is a family $\{A_i\xrightarrow{w_i} B\}_{i\in I}$ that is (part of) a colimit cone for the diagram consisting of objects $\{A_i\}_{i\in I}$ and spans in $R$ between them. Explicitly, for each span
\begin{equation}\label{eq:spanReg}
A_i\xleftarrow{x}D\xrightarrow{y} A_j
\end{equation}
in $R_{i,j}$, the square
\begin{equation}\label{diag:commSpan}
\begin{tikzpicture}[baseline=(current  bounding  box.center)]
\def \strx {1.2};
\def \stry {0.8};
\node (A1) at (-1*\strx,1*\stry) {$D$};
\node (A2) at (1*\strx,1*\stry) {$A_i$};
\node (B1) at (-1*\strx,-1*\stry) {$A_j$};
\node (B2) at (1*\strx,-1*\stry) {$B$};
\path[->,font=\scriptsize,>=angle 90]
(A1) edge node[above] {$x$} (A2)
(B1) edge node[below] {$w_j$} (B2)
(A1) edge node[left] {$y$} (B1)
(A2) edge node[right] {$w_i$} (B2);
\end{tikzpicture}
\end{equation}
commutes, and the quotient is a universal family with this property.
A {\it kernel} of an arbitrary family $\{A_i\xrightarrow{w_i} B\}_{i\in I}$, denoted $\mathrm{Ker}(\{w_i\})$, is the relation containing all spans of the form (\ref{eq:spanReg}) satisfying (\ref{diag:commSpan}).
\end{defn}
If a family $\{A_i\xrightarrow{w_i} B\}_{i\in I}$ quotients some relation, then it quotients its kernel. That is because adding more spans (such that (\ref{diag:commSpan}) commutes) to the colimit diagram does not change the colimit.
\begin{defn}
A family of maps $\{A_i\xrightarrow{w_i}B\}_{i\in I}$ 
in $\mathcal{V}$ is {\it jointly regular epi} if it 
is a quotient for some relation.
\end{defn}

\begin{example}
Cowedge components of a coend form a regular epi family. A functor $T:\mathcal{C}^\text{op}\times\mathcal{C}\rightarrow \mathcal{V}$ has a coend if and and only if  the relation on $\{T(C,C)\}_{C\in \mathcal{C}}$ formed by spans
\begin{equation}
T(C,C)\xleftarrow{T(f,C)}T(C',C)\xrightarrow{T(C',f)}T(C',C')
\end{equation}
for each $f:C\rightarrow C'$, has a quotient, and they are the same (up to isomorphism). This is a reformulation of obtaining a coend \cite{Lane1998} via a colimit.
\end{example}
\begin{example}
The same is true for an enriched coend. Let $\mathcal{V}$ be a locally small symmetric monoidal closed category, and $\mathcal{C}$ a $\mathcal{V}$-category. An enriched functor $T:\mathcal{C}^\text{op}\otimes\mathcal{C}\rightarrow \mathcal{V}$ can equivalently be seen as an endomodule on $\mathcal{C}$, given by actions
\begin{align}
\mathcal{C}(C',C'')\otimes T(C,C')&\xrightarrow{\lambda^{C'}_{CC''}}
T(C,C'')\\
T(C',C'')\otimes \mathcal{C}(C,C')&\xrightarrow{\rho^{C'}_{CC''}}
T(C,C'')\,.
\end{align}
It has a coend, defined as the quotient of the relation on $\{T(C,C)\}_{C\in \mathcal{C}}$ formed by spans
\begin{equation}
T(C,C)\xleftarrow{\rho^{C'}_{CC}\circ \sigma}\mathcal{C}(C,C')\otimes T(C',C)\xrightarrow{\lambda^{C}_{C'C'}}T(C',C')
\end{equation}
for each pair of objects $C,C'$. Note that this quotient is isomorphic to the one quotienting the relation formed from
\begin{equation}
T(C,C)\xleftarrow{\rho^{C'}_{CC}} T(C',C)\otimes\mathcal{C}(C,C')\xrightarrow{\lambda^{C}_{C'C'}\circ \sigma}T(C',C')
\end{equation}
since $\sigma$ is an isomorphism of spans constituting the colimit diagrams. 
\end{example}
\begin{example}\label{ex:moduleReg}
Module composition cocone components form regular epi family. Let $\mathcal{C},\mathcal{D},\mathcal{E}$ be small $\mathcal{V}$-enriched categories for a cocomplete monoidal category $\mathcal{V}$, together with a pair of modules $\mathcal{C}\xmrightarrow{M}\mathcal{D}\xmrightarrow{N}\mathcal{E}$. Fix objects $C\in \mathcal{C}$, $E\in\mathcal{E}$, and consider the relation on $\{ M(D,C)\otimes N(E,D)\}_{D\in \mathcal{D}}$ consisting of spans as in (\ref{eq:modComp}).
\begin{equation}\label{eq:modComp}
\begin{tikzpicture}[baseline=(current  bounding  box.center)]
\def \strx {1.2};
\def \stry {0.8};
\node (A1) at (0*\strx,2*\stry) {$M(D',C)\otimes \mathcal{D}(D,D')\otimes N(E,D)$};
\node (A2) at (-2*\strx,-1*\stry) {$M(D,C)\otimes N(E,D) $};
\node (B1) at (2*\strx,-1*\stry) {$M(D',C)\otimes N(E,D')$};
\path[->,font=\scriptsize,>=angle 90]
(A1) edge node[above left] {$\rho^{(M)}_{CDD'}\otimes 1$} (A2)
(A1) edge node[above right] {$1\otimes\lambda^{(N)}_{DD'E}$} (B1);
\end{tikzpicture}
\end{equation}
Its quotient is precisely the definition of the composite module, with quotient maps
\begin{equation}
M(D,C)\otimes N(E,D)\xrightarrow{w^{CE}_D}(N\circ_{\mathcal{D}}M)(E,C)\,.
\end{equation}
In particular, when $\mathcal{V}$ is symmetric closed, this is isomorphic to the enriched coend
\begin{equation}
\int^{D\in \mathcal{D}}M(D,C)\otimes N(E,D)\,.
\end{equation}
\end{example}

\begin{remark}
If $J=1$ the above definitions reduce to the definitions of (extremal, strong or regular \cite{Kelly1969}) epimorphisms. Furthermore, if $\mathcal{V}$ has coproducts, the induced map $\sum_i A_i \xrightarrow{w} B$ is extremal/strong epi if and only if $\{w_i\}_{i\in I}$ is a jointly extremal/strong epi family.  
\end{remark}
The regular case is examined in
\begin{prop}
In the category $\mathcal{V}$ with coproducts, a jointly regular epi family $\{w_i\}_{i\in I}$ induces a regular epi map $\sum_i A_i \xrightarrow{w} B$. The converse is true
 if for all parallel pairs $x,y:D\rightarrow \sum_i A_i$ the family 
 \begin{align}
 F_{xy}=\{p:P\rightarrow D|&\exists i,j,p_i:P\rightarrow A_i,
	 p_j:P\rightarrow A_j,\;\; \text{such that} \\
	 &x\circ p=\theta_i\circ p_i \;\;\text{and}\;\; y\circ p=\theta_i\circ p_j\}
 \end{align}
 is jointly epi.
\end{prop}
\begin{proof}
Considering the diagram
\begin{equation}
\begin{tikzpicture}[xscale=1.5, every node/.style={scale=1},baseline=(current  bounding  box.center)]
\node (D) at (-2,0) {$D$};
\node (Ai) at (-1,0) {$A_i$};
\node (SA) at (0,0) {$\sum A_i$};
\node (B) at (1,1) {$B$};
\node (C) at (1,-1) {$C$};
\path[->,font=\scriptsize,>=angle 90]
(D) edge node[above] {$x$} (Ai);
\path[->,font=\scriptsize,>=angle 90]
(Ai) edge node[above] {$\theta_i$} (SA);
\path[->,font=\scriptsize,>=angle 90]
(SA) edge node[below] {$w$} (B);
\path[->,font=\scriptsize,>=angle 90]
(SA) edge node[above] {$f$} (C);
\path[->,font=\scriptsize,>=angle 90]
(Ai) edge [bend left] node[above] {$w_i$} (B);
\path[->,font=\scriptsize,>=angle 90]
(Ai) edge [bend right] node[below] {$f_i$} (C);
\end{tikzpicture}
\end{equation}
it is easy to see that
\begin{equation}
\mathrm{Ker}(w)\subset \mathrm{Ker}(f)\;\implies\; \mathrm{Ker}\{w_i\}\subset \mathrm{Ker}\{f_i\}
\end{equation}
so given $f$ satisfying $\mathrm{Ker}(w)\subset \mathrm{Ker}(f)$, and using that $\{w_i\}$ is joint regular epi we get a unique factorization of $f$ through $w$, proving that $w$ is regular epi.

Conversely, given a regular epi $w$, and $f$ such that $\mathrm{Ker}\{w_i\}\subset \mathrm{Ker}\{f_i\}$, consider an arbitrary element of $\mathrm{Ker}(w)$,
$x,y:D\rightarrow\sum_i A_i$, that is $w\circ x=w\circ y$, and an arrow $p\in F_{xy}$.
Chasing diagrams gives
\begin{align}
w_i\circ p_i &=  w \circ \theta_i\circ p_i\\
			&=  w \circ x \circ p\\
			&=  w \circ y \circ p\\
			&=  w \circ \theta_j \circ p_j\\
			&=  w_j \circ p_j  
\end{align}
so $(p_i,p_j)\in \mathrm{Ker}(\{w_i\})$, and using the assumption for $f$, $(p_i,p_j)\in \mathrm{Ker}(\{f_i\})$. So we have
\begin{align}
f_i\circ p_i &=  f_j \circ p_j\\
f \circ \theta_i \circ p_i &=  f \circ \theta_j \circ p_j\\
f\circ x\circ p &=  f\circ y\circ p\,.
\end{align}
Using joint epiness of $F_{xy}$ we conclude that $(x,y)\in \mathrm{Ker}(f)$, and, because $w$ is regular epi, $f$ factors uniquely through it.
\end{proof}

\begin{remark}
As in the single epi case, any jointly regular epi family is automatically jointly strong epi.
The converse is true when $\mathcal{V}$ is familialy regular, a proof of a stronger statement is given in \cite{Street1984}.
\end{remark}

\begin{example}
In a preordered set $\mathcal{V}$ any family $\{A_i\xrightarrow{w_i}B\}_{i\in I}$ is jointly epi.
\end{example}

\begin{example}\label{ex:jRegPos}
In a poset $\mathcal{V}$ with arbitrary joins, a family $\{A_i\xrightarrow{w_i}B\}_{i\in I}$ is jointly extremal/strong/regular if and only if $B=\bigvee_i A_i$.
\end{example}